\documentclass{article}
\usepackage[utf8]{inputenc}

\pdfoutput=1
\usepackage{amsmath}
\usepackage{amsfonts}
\usepackage{amssymb}
\usepackage{centernot}
\usepackage{graphicx}
\usepackage{amsthm}
\usepackage{enumerate}
\usepackage{hyperref}

\usepackage{tikz}
\usetikzlibrary{decorations.markings}
\usetikzlibrary{arrows}
\providecommand{\U}[1]{\protect\rule{.1in}{.1in}}
\newtheorem{theorem}{Theorem}

\newtheorem{corollary}[theorem]{Corollary}

\newtheorem{definition}[theorem]{Definition}

\newtheorem{lemma}[theorem]{Lemma}

\newtheorem{proposition}[theorem]{Proposition}

\theoremstyle{remark}

\begin{document}

\title{Multivariate Alexander quandles, III. Sublinks}
\author{Lorenzo Traldi\\Lafayette College\\Easton, PA 18042, USA\\traldil@lafayette.edu
}
\date{ }
\maketitle

\begin{abstract}
If $L$ is a classical link then the multivariate Alexander quandle, $Q_A(L)$, is a substructure of the multivariate Alexander module, $M_A(L)$. In the first paper of this series we showed that if two links $L$ and $L'$ have $Q_A(L) \cong Q_A(L')$, then after an appropriate re-indexing of the components of $L$ and $L'$, there will be a module isomorphism $M_A(L) \cong M_A(L')$ of a particular type, which we call a ``Crowell equivalence.'' In the present paper we show that $Q_A(L)$ (up to quandle isomorphism) is a strictly stronger link invariant than $M_A(L)$ (up to re-indexing and Crowell equivalence). This result follows from the fact that $Q_A(L)$ determines the $Q_A$ quandles of all the sublinks of $L$, up to quandle isomorphisms.


\end{abstract}

\section{Introduction}

In this paper we continue to investigate the connections between two kinds of algebraic invariants of classical links, Alexander modules and quandles. Before stating our results we recall some properties of these invariants.

The basic theory of multivariate Alexander modules was developed over a period of sixty years or so, starting with Alexander's introduction of the reduced (one-variable) polynomial invariants that bear his name \cite{A}. Multivariate versions of Alexander's polynomial invariants were studied by Fox and his students; see \cite{F,To} and works cited there. If $L=K_1 \cup \dots \cup K_{\mu}$ is a classical link of $\mu$ components, then these multivariate invariants are called the elementary ideals of $L$; they are ideals of the ring $\Lambda_{\mu}=\mathbb{Z}[t_1^{\pm1},\dots,t_{\mu}^{\pm1}]$ of Laurent polynomials in the variables $t_1,\dots,t_{\mu}$, with integer coefficients. The elementary ideals are associated with homology groups $H_1(\widetilde X)$ and $H_1(\widetilde X,F)$, where $\widetilde X$ is the universal abelian covering space of $X=\mathbb S ^3 - L$, $F$ is the cover's fiber, and the homology groups are considered as $\Lambda_{\mu}$-modules with scalar multiplications derived from the cover's deck transformations. (In particular, scalar multiplication by $t_i$ is associated with the element of $\pi_1(X)$ represented by a meridian of $K_i$.)  Crowell \cite{C1, C3, CS} observed that the long exact homology sequence of $(\widetilde X,F)$ yields a short exact \emph{link module sequence}
\begin{equation}
\label{lms}
0 \to \ker \phi_L \xrightarrow{\psi} M_A(L) \xrightarrow{\phi_L} I_{\mu} \to 0 \text{,}
\end{equation}
where $M_A(L)$ is the \emph{Alexander module} of $L$ (i.e., $H_1(\widetilde X,F)$), $I_{\mu}$ is the \emph{augmentation ideal} of $\Lambda_{\mu}$ (i.e., the ideal generated by $\{t_1-1, \dots, t_{\mu}-1\}$), and $\psi$ is the inclusion map. The module $\ker \phi_L$ (i.e., $H_1(\widetilde X)$) is the \emph{Alexander invariant} of $L$. We refer the reader to \cite{F, H, L, Ro} for more information, with the warning that terminology in the references is not consistent. For instance, Lickorish \cite{L} used the term ``Alexander module'' for what we call the reduced (or one-variable) version of the Alexander invariant, obtained by replacing $\widetilde X$ with the infinite cyclic cover $X_{\infty}$ of $X$. ($X_{\infty}$ is also called the ``total linking number'' cover of $X$.) The corresponding homology groups are modules over the ring $\Lambda=\mathbb{Z}[t^{\pm1}]$ of Laurent polynomials in the variable $t$, with integer coefficients. To be clear, we always use ``reduced'' to refer to modules and maps associated with the infinite cyclic cover rather than the universal abelian cover, and we sometimes use ``multivariate'' to refer to (\ref{lms}) and the $\Lambda_{\mu}$-modules that appear in it.

Most discussions of these invariants focus on either the Alexander invariant or the Alexander module, rather than the link module sequence. For knots there is little significant difference, because the link module sequence splits. In general, though, it seems possible that the link module sequence is a more sensitive link invariant than either the Alexander invariant or the Alexander module. (We do not know of examples that confirm this possibility.) For this reason, we focus on the sequence rather than either of the individual modules. 

The link module sequence of $L$ is determined by the homomorphism $\phi_L:M_A(L) \to I_{\mu}$. This observation motivates the following definition of Crowell \cite{C1}.

\begin{definition}
\label{Crowelleq}
The link module sequences of two $\mu$-component links $L$ and $L'$ are \emph{equivalent} if there is a $\Lambda_{\mu}$-linear isomorphism $f:M_A(L) \to M_A(L')$ such that $\phi_L=\phi_{L'} f:M_A(L) \to I_{\mu}$.
\end{definition}

If $f$ satisfies Definition \ref{Crowelleq} then we say that $f$ is a \emph{Crowell equivalence}, and $L$ and $L'$ are \emph{Crowell equivalent}. 
\begin{proposition} These three properties hold.
\label{lmsprop}
\begin{enumerate} [1.]
    \item Re-indexing the components of $L=K_1 \cup \dots \cup K_{\mu}$ may result in a link that is not Crowell equivalent to $L$. 
    \item If $L=K_1 \cup \dots \cup K_{\mu}$ and $L'=K'_1 \cup \dots \cup K'_{\mu}$ are Crowell equivalent, they may have sublinks $L=K_1 \cup \dots \cup K_{m}$ and $L'=K'_1 \cup \dots \cup K'_{m}$ that are not Crowell equivalent.
    \item If $L=K_1 \cup \dots \cup K_{\mu}$ and $L'=K'_1 \cup \dots \cup K'_{\mu}$ are not Crowell equivalent, then their reduced (one-variable) link module sequences may be equivalent.
\end{enumerate}
\end{proposition}
\begin{proof}
The first property is illustrated by any link whose multivariate Alexander polynomial is changed by re-indexing the variables $t_1,\dots,t_{\mu}$. Examples illustrating the second property are given in Sec. \ref{examples} below, and examples illustrating the third property were mentioned in \cite{mvaq1}. 
\end{proof}

Examples of property 2 of Proposition \ref{lmsprop} are not hard to find. While preparing this paper we looked for pairs of links with the same elementary ideals whose sublinks are distinguished by their own elementary ideals, and then checked to see if the links are Crowell equivalent. The first three such pairs we analyzed, $\{5^2_1,7^2_8\}$, $\{7^2_3,9^2_{46}\}$ and $\{6^3_2,9^3_{18}\}$ in Rolfsen's table \cite[Appendix C]{Ro}, all exemplify property 2 of Proposition \ref{lmsprop}.

Quandles are algebraic invariants introduced to classical knot theory by Joyce \cite{J} and Matveev \cite{M} in the early 1980s. If $L$ is a link then the fundamental quandle $Q(L)$ is a subset of the link group $\pi_1(\mathbb S ^3 - L)$; it is the union of the conjugacy classes of the meridians, considered as an algebraic system with an operation defined by conjugation in $\pi_1(\mathbb S ^3 - L)$. Joyce and Matveev both observed that the reduced (one-variable) version of the Alexander module can be considered as a quandle, with the quandle operation derived from the module structure. This kind of quandle is usually called an ``Alexander quandle'' in the literature; we refer to it as a \emph{standard} Alexander quandle. Notice that a standard Alexander quandle is an entire $\Lambda$-module; in contrast, a link's fundamental quandle is a proper subset of the link's group.

A different kind of quandle associated to an Alexander module was introduced in the first paper in this series \cite{mvaq1}. The fundamental multivariate Alexander quandle $Q_A(L)$ is a subset of the multivariate Alexander module $M_A(L)$. If $L=K_1 \cup \dots \cup K_{\mu}$ then $Q_A(L)$ has $\mu$ orbits, one for each component $K_i$. It follows that $Q_A(L)$ determines the number $\mu$; however, there is no way to tell which orbit corresponds to which component, using only information from $Q_A(L)$ itself. That is, if we permute the component indices in $L$ then $Q_A(L)$ is unchanged.

\begin{definition}
Suppose $L=K_1 \cup \dots \cup K_{\mu}$ and $L'=K'_1 \cup \dots \cup K'_{\mu}$ are classical links, $f:Q_A(L) \to Q_A(L')$ is a quandle isomorphism, and for each $i \in \{1, \dots, \mu\}$, the image under $f$ of the $K_i$ orbit of $Q_A(L)$ is the $K'_i$ orbit of $Q_A(L')$. Then we say that $L$ and $L'$ are \emph{indexed compatibly with }$f$. \end{definition}

As shown in \cite{mvaq1}, once we know which $K_i$ corresponds to each orbit in $Q_A(L)$ we can use $Q_A(L)$ to construct a presentation of $M_A(L)$ as a $\Lambda_{\mu}$-module. This module presentation also determines the map $\phi_L$. We state some consequences of these properties as a proposition, for ease of reference.

\begin{proposition}
\label{backref}
(\cite{mvaq1}) Suppose $L=K_1 \cup \dots \cup K_{\mu}$ and $L'=K'_1 \cup \dots \cup K'_{\mu'}$ are classical links, and $f:Q_A(L) \to Q_A(L')$ is an isomorphism. Then these three properties hold.
\begin{enumerate} [(a)]
    \item $\mu=\mu'$.
    \item The components of $L$ and $L'$ can be re-indexed compatibly with $f$.
    \item Once the components are indexed compatibly with $f$, $f$ will extend to a Crowell equivalence $g:M_A(L) \to M_A(L')$.
\end{enumerate}
\end{proposition}

Proposition \ref{backref} includes the implication $2 \implies 3$ of the main theorem of \cite{mvaq1}. The proposition is useful because quandles are rather intractable, compared to other kinds of algebraic structures. When trying to determine whether two links $L$ and $L'$ have $Q_A(L) \cong Q_A(L')$, instead of working directly with $Q_A(L)$ and $Q_A(L')$ it is much easier to first determine whether $M_A(L) \cong M_A(L')$, perhaps after re-indexing of the links' components, and if so, to then determine whether any $\Lambda_{\mu}$-module isomorphism between $M_A(L)$ and $M_A(L')$ is a Crowell equivalence. After establishing that $L$ and $L'$ are Crowell equivalent, one can then look for Crowell equivalences that map $Q_A(L)$ onto $Q_A(L')$. An example of this sort of analysis is given in Secs.\ \ref{examples} and \ref{moreexamples}. 

We are now ready to state the two central results of the present paper. The first result is that $Q_A(L)$ determines the $Q_A$ quandles of all the sublinks of $L$. To state this property precisely, we use the convention that if $L=K_1 \cup \dots \cup K_{\mu}$ and $S \subseteq \{1,  \dots, \mu \}$ then $L_S$ denotes the sublink of $L$ consisting of components with indices from $S$.

\begin{theorem}
\label{main}
Suppose $L=K_1 \cup \dots \cup K_{\mu}$, $L'=K'_1 \cup \dots \cup K'_{\mu'}$ and $f:Q_A(L) \cong Q_A(L')$. Proposition \ref{backref} tells us that $\mu=\mu'$, and $L$ and $L'$ may be re-indexed compatibly with $f$. After such a re-indexing, it will be true that $Q_A(L_S) \cong Q_A(L'_S)$ $\forall S \subseteq \{1,  \dots, \mu \}$.
\end{theorem}

Combining Theorem \ref{main} with part 2 of Proposition \ref{lmsprop}, we obtain the following.

\begin{corollary}
\label{nocon}
Aside from indexing of link components, the fundamental multivariate Alexander quandle is a strictly stronger link invariant than Crowell's link module sequence. To be explicit: if $Q_A(L) \cong Q_A(L')$, then there are re-indexed versions of $L$ and $L'$ that are Crowell equivalent; but if $L$ and $L'$ are Crowell equivalent, they may have $Q_A(L) \centernot \cong Q_A(L')$.
\end{corollary}

This result may be surprising because according to the definition in \cite{mvaq1}, the quandle operations $\triangleright, \triangleright^{-1}$ of $Q_A(L)$ are given by formulas involving the Crowell map of $L$. It follows that a Crowell equivalence $f:M_A(L) \to M_A(L')$ yields a precise correspondence between the formulas that define $Q_A(L)$ and $Q_A(L')$. Nevertheless, it does not follow that $f$ defines a quandle isomorphism between $Q_A(L)$ and $Q_A(L')$, because $f$ can match the formulas without matching the elements of the two quandles. In Sec.\ 4 we verify that this is the situation for the examples of Sec.\ \ref{examples}.

Corollary \ref{nocon} completes the basic theory of multivariate Alexander quandles, by contradicting the converse of the implication denoted $2 \implies 3$ in \cite{mvaq1}. Note the contrast with \cite{mvaq2}, where we showed that the involutory medial quandle $IMQ(L)$ is equivalent (as a link invariant) to a simplified version of the link module sequence (\ref{lms}). 

Our results raise the possibility of strengthening some invariants of classical links associated with multivariate Alexander modules -- including the Alexander polynomials, Arf invariant, determinant, elementary ideals, linking numbers, Milnor $\bar \mu$-invariants, and others -- to reflect their connection with multivariate Alexander quandles. In particular, we wonder whether it is possible to use $Q_A(L)$ to produce a refined Alexander polynomial that distinguishes the links of Sec.\ \ref{examples}.

\section{Proposition \ref{lmsprop}}
\label{examples}

In this section we present a Crowell equivalence between Whitehead's link $W$ and the link denoted $7^2_8$ in Rolfsen's table \cite{Ro}. This equivalence gives us the second property of Proposition \ref{lmsprop}, because the components of the two links are not Crowell equivalent. Both components of $W$ are trivial, and one component of $7^2_8$ is a trefoil. As the Alexander polynomial of the trefoil is nontrivial, the Alexander module of the trefoil is not isomorphic to the Alexander module of the trivial knot. 

First, we recall how to obtain presentations of Alexander modules from link diagrams. Given a diagram $D$ of $L=K_1 \cup \dots \cup K_{\mu}$, let $A(D)$ and $C(D)$ be the sets of arcs and crossings of $D$. Let $\kappa_D:A(D) \to \{1, \dots, \mu \}$ be the function with $\kappa_D(a)=i$ if $a$ belongs to the image of $K_i$ in $D$, let $\Lambda_{\mu}^{A(D)}$ and $\Lambda_{\mu}^{C(D)}$ be the free $\Lambda_{\mu}$-modules on the sets $A(D)$ and $C(D)$, and let $\rho_D:\Lambda_{\mu}^{C(D)} \to \Lambda_{\mu}^{A(D)}$ be the $\Lambda_{\mu}$-linear map given by
\[
\rho_D(c)=(1-t_{\kappa_D(a_{\mathrm{left}})})a_{\mathrm{over}}+t_{\kappa_D(a_{\mathrm{over}})}a_{\mathrm{right}}-a_{\mathrm{left}}
\]
whenever $c \in C(D)$ is a crossing of $D$ as indicated in Fig.\ \ref{crossfig}.

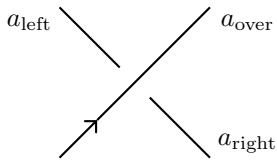
\begin{figure} [bht]
\centering
\begin{tikzpicture} [>=angle 90]
\draw [thick] (1,1) -- (-0.5,-0.5);
\draw [thick] [<-] (-0.5,-0.5) -- (-1,-1);
\draw [thick] (-1,1) -- (-.2,0.2);
\draw [thick] (0.2,-0.2) -- (1,-1);
\node at (1.5,0.8) {$a_{\mathrm{over}}$};
\node at (-1.4,0.8) {$a_{\mathrm{left}}$};
\node at (1.5,-0.8) {$a_{\mathrm{right}}$};
\end{tikzpicture}
\caption{A crossing.}
\label{crossfig}
\end{figure}

Then a presentation of the Alexander module $M_A(L)$ is given by an exact sequence 
\begin{equation*}
\Lambda_{\mu}^{C(D)} \xrightarrow{\rho_D} \Lambda_{\mu}^{A(D)} \xrightarrow{\gamma_D} M_A(L) \to 0 \text{,}
\end{equation*}
and the Crowell map $\phi_L:M_A(L) \to I_{\mu}$ is given by $\phi_L\gamma_D(a)=t_{\kappa_D(a)}-1$ $\forall a \in A(D)$. If $D$ and $D'$ are two diagrams of the same link $L$, then there is an isomorphism between the two resulting instances of the Alexander module $M_A(L)$, which is compatible with the map $\phi_L$. This invariance property follows from the fact that the link module sequence is derived from homology groups associated with the universal abelian cover of $\mathbb S ^3 - L$; it may also be verified using the Reidemeister moves, as in \cite{mvaq1}.

\subsection{Whitehead's link}
\label{wsec}

A diagram $D$ of Whitehead's link appears in Fig.\ \ref{wfig}. The set of arcs is $A(D)=\{a_1,a_2,a_3,a_4,a_5\}$, and the set of crossings is $C(D)=\{c_1,c_2,c_3,c_4,c_5\}$. To avoid cluttering the figure, crossing indices are not indicated explicitly. Instead, we adopt the convention that each crossing shares the index of the underpassing arc directed into that crossing; for instance, the central crossing of Fig.\ \ref{wfig} is $c_5$. The component function $\kappa_D:A(D) \to \{1,2\}$ has $\kappa_D(a_2)=\kappa_D(a_4)=\kappa_D(a_5)=1$ and $\kappa_D(a_1)=\kappa_D(a_3)=2$. 
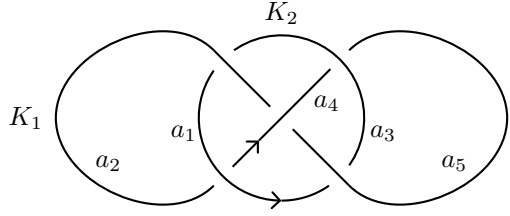
\begin{figure} [bt]
\centering
\begin{tikzpicture} 
\draw [thick] (-0.9,0.9) -- (-0.15,0.15);
\draw [thick] (0.15,-0.15) -- (0.9,-0.9);
\draw [thick] (0.65,0.65) -- (-0.3,-0.3);
\draw [thick] [->] [>=angle 90] (-0.65,-0.65) -- (-0.3,-0.3);
\draw [->] [>=angle 90] [thick, domain=-215:-90] plot ({(1.1)*cos(\x)}, {(1.1)*sin(\x)});
\draw [thick, domain=-90:-55] plot ({(1.1)*cos(\x)}, {(1.1)*sin(\x)});
\draw [thick, domain=-35:125] plot ({(1.1)*cos(\x)}, {(1.1)*sin(\x)});
\draw [thick] (0.9,0.9) to [out=45, in=90] (3,0);
\draw [thick] (3,0) to [out=-90, in=-45] (0.9,-0.9);
\draw [thick] (-0.9,0.9) to [out=135, in=90] (-3,0);
\draw [thick] (-3,0) to [out=-90, in=-135] (-0.9,-0.9);
\node at (0,1.4) {$K_2$};
\node at (-3.4,0) {$K_1$};
\node at (-1.3,-0.2) {$a_1$};
\node at (-2.3,-0.6) {$a_2$};
\node at (0.6,0.2) {$a_4$};
\node at (2.3,-0.6) {$a_5$};
\node at (1.35,-0.2) {$a_3$};
\end{tikzpicture}
\caption{Whitehead's link, $W$.}
\label{wfig}
\end{figure}

We proceed to simplify the presentation of $M_A(W)$ derived from $D$. The relations $\gamma_D\rho_D(c_3)$ $=0$ and $\gamma_D\rho_D(c_4)=0$ tell us that $\gamma_D(a_1)=(1-t_2)\gamma_D(a_2)+t_1\gamma_D(a_3)$ and $\gamma_D(a_4)=(1-t_1)\gamma_D(a_3)+t_2\gamma_D(a_5)$.
Using these formulas, we see that $M_A(W)$ is generated by $\gamma_D(a_2),\gamma_D(a_3)$ and $\gamma_D(a_5)$, subject to the following relations:
\begin{align*}
& \gamma_D\rho_D(c_1)=(1-t_2)\gamma_D(a_5)+t_1\gamma_D(a_3)-\gamma_D(a_1)
\\
& \qquad \qquad =(1-t_2)(\gamma_D(a_5)-\gamma_D(a_2))=0
\\
& \gamma_D\rho_D(c_2)=(1-t_1)\gamma_D(a_1)+t_2\gamma_D(a_2)-\gamma_D(a_4)
\\
& \qquad\qquad =(1-t_1+t_1t_2)\gamma_D(a_2)-(1-t_1)^2\gamma_D(a_3)-t_2 \gamma_D(a_5)=0
\\
& \gamma_D\rho_D(c_5)=(1-t_1)\gamma_D(a_4)+t_1\gamma_D(a_5)-\gamma_D(a_2)
\\
& \qquad\qquad =-\gamma_D(a_2)+(1-t_1)^2\gamma_D(a_3)+(t_1+t_2-t_1t_2)\gamma_D(a_5)=0.
\end{align*}
Rewriting the relations in terms of $\gamma_D(a_2),\gamma_D(a_3)$ and $x=\gamma_D(a_5)-\gamma_D(a_2)$, and adding the third relation to the second, we obtain the following:
\begin{align*}
& (1-t_2)x=0 \\   
& t_1(1-t_2)x=0 \\
& (-t_1t_2+t_1+t_2-1) \gamma_D(a_2) +(1-t_1)^2\gamma_D(a_3) +(t_1+t_2-t_1t_2)x=0.
\end{align*}
Subtracting the second relation from the third, and noticing that the second relation follows from the first, we conclude that these two relations suffice:
\begin{align*}
& (1-t_2)x=0 \\   
& (-t_1t_2+t_1+t_2-1) \gamma_D(a_2) +(1-t_1)^2\gamma_D(a_3) + t_2x=0.
\end{align*}
The second of these relations tells us that 
\[
x=t_2^{-1} \cdot ( (1-t_1)(1-t_2) \gamma_D(a_2) - (1-t_1)^2\gamma_D(a_3) ).
\]
It follows that $M_A(W)$ is generated by $\gamma_D(a_2)$ and $\gamma_D(a_3)$, subject to the single relation
\begin{equation}
\label{rel}
(1-t_2) \cdot ((1-t_1)(1-t_2) \gamma_D(a_2)- (1-t_1)^2 \gamma_D(a_3))=0.
\end{equation}

\subsection{The link \texorpdfstring{$7^2_8$}{7}}
\label{lsec}

Let $E$ be the diagram of the link $L=7^2_8$ depicted in Fig.\ \ref{otherfig}. Again, if $a_i$ is the underpassing arc oriented into a crossing, then the crossing is denoted $c_i$.

\begin{figure} [bht]
\centering
\begin{tikzpicture} 
\draw [thick] (-1,6) -- (1,6);
\draw [thick] (1,6) to [out=0, in=90] (2,5.2);
\draw [thick] (2,5.2) -- (2,4.2);
\draw [thick] (-1,6) to [out=180, in=90] (-2,5.2);
\draw [thick] (-2,5.2) -- (-2,4.2);
\draw [thick] (-0.8,4) -- (2.2,4);
\draw [thick] (2,3.8) -- (2,1.8);
\draw [thick] (2.2,4) to [out=0, in=0] (2.2,2);
\draw [thick] (1,1) to [out=0, in=-90] (2,1.8);
\draw [thick] (-1,1) to [out=180, in=-90] (-2,1.8);
\draw [thick] [->] [>=angle 90] (-1,1) -- (1,1);
\draw [thick] (-2,3.8) -- (-2,1.8);
\draw [thick] [->] [>=angle 90] (1,4.2) to [out=90, in=0] (0,5);
\draw [thick] (-1,4.2) to [out=90, in=180] (0,5);
\draw [thick] (-1,4.2) -- (-1,3.8);
\draw [thick] (-1,3.8) to [out=-90, in = 180] (1.8,2);
\fill [white] (0.3,2.5) rectangle (-0.6,1.5);
\draw [thick] (1,3.8) to [out=-90, in = 0] (-1.8,2);
\draw [thick] (-1.2,4) -- (-2.2,4);
\draw [thick] (-2.2,4) to [out=180, in=180] (-2.2,2);
\node at (2.4,5.5) {$K_2$};
\node at (3.2,3) {$K_1$};
\node at (0,1.2) {$a_1$};
\node at (-3.1,3) {$a_2$};
\node at (0,5.7) {$a_3$};
\node at (-0.9,1.8) {$a_4$};
\node at (0,4.7) {$a_5$};
\node at (1.2,1.8) {$a_6$};
\node at (0,3.7) {$a_7$};
\end{tikzpicture}
\caption{$L=7^2_8$.}
\label{otherfig}
\end{figure}

We eliminate four generators of $M_A(L)$ using the following formulas:
\begin{align*}
& \text{from } \rho_E(c_3): \gamma_E(a_3)=t_1^{-1} \cdot ((t_2-1)\gamma_E(a_2)+\gamma_E(a_1))\\
& \text{from } \rho_E(c_2): \gamma_E(a_4)=(1-t_1)\gamma_E(a_1)+t_2\gamma_E(a_2)\\
& \text{from } \rho_E(c_6): \gamma_E(a_6)=(1-t_1)\gamma_E(a_1)+t_2\gamma_E(a_7)\\
& \text{from } \rho_E(c_5): \gamma_E(a_5)=(1-t_1)\gamma_E(a_4)+t_1\gamma_E(a_6)\\
& \qquad\qquad\qquad\qquad =(1-t_1)\gamma_E(a_1)+(1-t_1)t_2\gamma_E(a_2)+t_1t_2\gamma_E(a_7).
\end{align*}

It follows that $M_A(L)$ is generated by $\gamma_E(a_1),\gamma_E(a_2)$ and $\gamma_E(a_7)$, subject to the following relations:
\begin{align*}
& \gamma_E\rho_E(c_1)=(1-t_2)\gamma_E(a_7)+t_1\gamma_E(a_3)-\gamma_E(a_1)
\\
& \quad =(1-t_2)(\gamma_E(a_7)-\gamma_E(a_2))=0
\\
& \gamma_E\rho_E(c_4)=(1-t_1)\gamma_E(a_7)+t_1\gamma_E(a_5)-\gamma_E(a_4)
\\
& \quad =(-1+2t_1-t_1^2)\gamma_E(a_1)+t_2(t_1-t_1^2-1)\gamma_E(a_2)+(1-t_1+t_1^2 t_2)\gamma_E(a_7)=0
\\
& \gamma_E\rho_E(c_7)=(1-t_1)\gamma_E(a_5)+t_1\gamma_E(a_2)-\gamma_E(a_7)\\
& \quad = (1-t_1)^2\gamma_E(a_1) +( (1-t_1)^2t_2+t_1)\gamma_E(a_2)+(t_1t_2-t_1^2t_2 -1) \gamma_E(a_7)=0.
\end{align*}

Notice that $\gamma_E\rho_E(c_7)=-t_1\gamma_E\rho_E(c_1)-\gamma_E\rho_E(c_4)$, so we can ignore the relation $\gamma_E\rho_E(c_7)=0$. $M_A(L)$ is generated by $\gamma_E(a_2)$, $y=\gamma_E(a_1)+\gamma_E(a_2)-\gamma_E(a_7)$ and $z=t_1(\gamma_E(a_2)-\gamma_E(a_7))$. When we rewrite $\gamma_E\rho_E(c_1)$ and $\gamma_E\rho_E(c_4)$ in terms of these generators we obtain the following:
\begin{align*}
& (1-t_2)z=0\\
& (1-t_1)(1-t_2)\gamma_E(a_2)-(1-t_1)^2y+(t_1(1-t_2)-1)z=0.
\end{align*}
As the first relation tells us that $(1-t_2)z=0$, the second relation can be replaced with $(1-t_1)(1-t_2)\gamma_E(a_2)-(1-t_1)^2y-z=0$, which tells us that $z=(1-t_1)(1-t_2)\gamma_E(a_2)-(1-t_1)^2y$. We conclude that $M_A(L)$ is generated by $\gamma_E(a_2)$ and $y$, subject to the single relation 
\[
(1-t_2) \cdot((1-t_1)(1-t_2)\gamma_E(a_2)-(1-t_1)^2y)=0.
\]

Comparing this with (\ref{rel}), we see that there is an isomorphism $f:M_A(W) \to M_A(L)$ of $\Lambda_2$-modules, given by $f (\gamma_D(a_2))=\gamma_E(a_2)$ and $f (\gamma_D(a_3))=y$. As $\phi_W(\gamma_D(a_2))=t_1-1=\phi_L(\gamma_E(a_2))$, $\phi_W (\gamma_D(a_3))=t_2-1$ and $\phi_L(y)=\phi_L(\gamma_E(a_1)+\gamma_E(a_2)-\gamma_E(a_7))=t_2-1+t_1-1-(t_1-1)=t_2-1$, $f$ is a Crowell equivalence.

\section{Proof of Theorem \ref{main}}
\label{mainproof}

Theorem \ref{main} is proven using the relationship between the Alexander modules of a link $L=K_1 \cup \dots \cup K_{\mu}$ and its sublink $L-K_{\mu}=K_1 \cup \dots \cup K_{\mu-1}$. This relationship can be traced back almost 70 years, to Torres' Ph.D. dissertation on the Alexander polynomial \cite{To}. The theory was elaborated in two later dissertations, of Sato \cite{Sa} and the present author \cite{Tr}. A thorough account is given by Hillman \cite[Chap.\ 5]{H}.

Every diagram $D$ of $L$ yields a diagram $D_{\mu}$ of $L-K_{\mu}$, obtained by removing all the arcs of $D$ belonging to the image of $K_{\mu}$, and replacing each crossing of $K_{\mu}$ over another component with a trivial crossing, as indicated in Fig.\ \ref{subfig}. There is a natural way to identify elements of $A(D_{\mu})$ with elements of $\{ a \in A(D) \mid \kappa_D(a)<\mu\}$; if $a \in A(D_{\mu})$ then we also use $a$ to denote the corresponding element of $A(D)$. Similarly, if $c$ is a crossing of $D$ in which $K_{\mu}$ is not the underpassing component then we also use $c$ to denote the corresponding crossing of $D_{\mu}$.

\begin{figure}[bht]
\centering
\begin{tikzpicture} [>=angle 90]
\draw [thick] (1,1) -- (-0.5,-0.5);
\draw [thick] [<-] (-0.5,-0.5) -- (-1,-1);
\draw [thick] (-1,1) -- (-.2,0.2);
\draw [thick] (0.2,-0.2) -- (1,-1);
\draw [thick] (5.6,0) -- (6,0.4);
\draw [thick] (5,1) -- (5.6,0.4);
\draw [thick] (6.6,-0.1) to [out=90, in=45] (6,0.4);
\draw [thick] (6.6,-0.1) to [out=-90, in=-225] (6.7,-0.7);
\draw [thick] (6.7,-0.7) -- (7,-1);
\draw [thick] (6,0) to [out=-45, in=0] (5.8,-0.8);
\draw [thick] (5.6,0) to [out=225, in=180] (5.8,-0.8);
\node at (1.4,0.8) {$K_{\mu}$};
\node at (-1.4,0.8) {$a_{\mathrm{left}}$};
\node at (1.5,-0.8) {$a_{\mathrm{right}}$};
\node at (4.6,0.8) {$a_{\mathrm{left}}$};
\node at (7.5,-0.8) {$a_{\mathrm{right}}$};
\end{tikzpicture}
\caption{A crossing of another component under $K_{\mu}$ in $D$ is replaced by a trivial crossing in $D_{\mu}$.}
\label{subfig}
\end{figure}
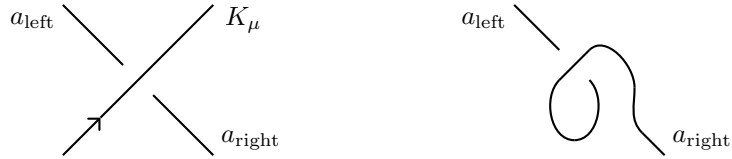
Notice that if $c$ is a crossing as pictured in Fig.\ \ref{subfig}, then the indicated ``left'' and ``right'' designations may not be accurate for $c$ in $D_{\mu}$. Depending on the orientation of that link component, $\rho_{D_{\mu}}(c)$ is either
\[
(1-t_{\kappa_{D_{\mu}}(a_{\mathrm{left}})})a_{\mathrm{right}}+t_{\kappa_{D_{\mu}}(a_{\mathrm{right}})}a_{\mathrm{right}}-a_{\mathrm{left}}=a_{\mathrm{right}}-a_{\mathrm{left}} 
\]
\[
\text{or   } (1-t_{\kappa_{D_{\mu}}(a_{\mathrm{right}})})a_{\mathrm{right}}+t_{\kappa_{D_{\mu}}(a_{\mathrm{right}})}a_{\mathrm{left}}-a_{\mathrm{right}}=t_{\kappa_{D_{\mu}}(a_{\mathrm{right}})} \cdot(a_{\mathrm{left}}-a_{\mathrm{right}}).
\]
Observe that in the second case, the submodule $\rho_{D_{\mu}}(\Lambda_{\mu-1}^{C(D_{\mu})})$ of $\Lambda_{\mu-1}^{A(D_{\mu})}$ is not changed if $\rho_{D_{\mu}}(c)$ is replaced with $a_{\mathrm{right}}-a_{\mathrm{left}}$.

Let $\pi:\Lambda_{\mu} \to \Lambda_{\mu-1}$ be the ring homomorphism given by $\pi(t_i)=t_i$ when $i<\mu$, and $\pi(t_{\mu})=1$. Then every $\Lambda_{\mu-1}$-module $M$ can also be considered as a $\Lambda_{\mu}$-module via $\pi$; that is, $\lambda \cdot x = \pi(\lambda) \cdot x$ $\forall \lambda \in \Lambda_{\mu}$ $\forall x \in M$. We use this observation in the following.

\begin{proposition}
\label{induct}
Let $D$ be a diagram of $L$, and $D_{\mu}$ the corresponding diagram of $L-K_{\mu}$. Let $N$ be the $\Lambda_{\mu}$-submodule of $M_A(L)$ generated by the set
\[
T=(t_{\mu}-1)M_A(L) \cup \{\gamma_D(a) \mid a \in A(D)\text{ and } \kappa_D(a)=\mu \}.
\]
Then there is an epimorphism $p:M_A(L) \to M_A(L-K_{\mu})$ of $\Lambda_{\mu}$-modules, with $\ker p = N$ and $p(\gamma_D(a))=\gamma_{D_{\mu}}(a)$ $\forall a \in A(D_{\mu})$.
\end{proposition}
\begin{proof}
Let $P:\Lambda_{\mu}^{A(D)} \to \Lambda_{\mu-1}^{A(D_{\mu})}$ be the $\Lambda_{\mu}$-linear map with $P(a)=a$ $\forall a \in A(D_{\mu})$, and $P(a)=0$ $\forall a \in A(D)$ with $\kappa_D(a)=\mu$. We claim that $\ker (\gamma_{D_{\mu}}P)$ contains $\rho_D(c)$ for every crossing $c$ of $D$. 

Suppose $c$ is a crossing of $D$ that does not involve $K_{\mu}$. Then $c$ also appears in the diagram $D_{\mu}$, and $\rho_{D_{\mu}}(c)$ has the same form as $\rho_D(c)$; hence $\gamma_{D_{\mu}}P(\rho_D(c))=\gamma_{D_{\mu}}(\rho_{D_{\mu}}(c))=0$ in $M_A(L-K_{\mu})$. 

Suppose $c$ is a crossing of $D$ as pictured in Fig.\ \ref{crossfig}, in which $K_{\mu}$ is the underpassing component. Then
\[
\gamma_{D_{\mu}}P(\rho_D(c))=
\gamma_{D_{\mu}}P((1-t_{\kappa_D(a_{\mathrm{left}})})a_{\mathrm{over}}+t_{\kappa_D(a_{\mathrm{over}})}a_{\mathrm{right}}-a_{\mathrm{left}})
\]
\[
=(1-t_{\mu})\gamma_{D_{\mu}}(a_{\mathrm{over}})+t_{\kappa_D(a_{\mathrm{over}})} \cdot 0 -0=\gamma_{D_{\mu}}(a_{\mathrm{over}})-\gamma_{D_{\mu}}(a_{\mathrm{over}})+0-0=0.
\]

Now, suppose $c$ is a crossing of $D$ as pictured on the left in Fig.\ \ref{subfig}, in which $K_{\mu}$ is the overpassing component, and not the underpassing component. Then
\[
\gamma_{D_{\mu}}P(\rho_D(c))=
\gamma_{D_{\mu}}P((1-t_{\kappa_D(a_{\mathrm{left}})})a_{\mathrm{over}}+t_{\kappa_D(a_{\mathrm{over}})}a_{\mathrm{right}}-a_{\mathrm{left}})
\]
\[
=(1-t_{\kappa_{D_{\mu}}(a_{\mathrm{left}})}) \cdot 0+t_{\mu}\gamma_{D_{\mu}}(a_{\mathrm{right}})-\gamma_{D_{\mu}}(a_{\mathrm{left}})=\gamma_{D_{\mu}}(a_{\mathrm{right}})-\gamma_{D_{\mu}}(a_{\mathrm{left}}).
\]
As noted before the statement of the proposition, $a_{\mathrm{right}}-a_{\mathrm{left}} \in \rho_{D_{\mu}}(\Lambda_{\mu-1}^{C(D_{\mu})})$; hence $\gamma_{D_{\mu}}(a_{\mathrm{right}})-\gamma_{D_{\mu}}(a_{\mathrm{left}})=\gamma_{D_{\mu}}(a_{\mathrm{right}} - a_{\mathrm{left}})=0$. 

Thus $\rho_D(c) \in \ker (\gamma_{D_{\mu}}P)$ $\forall c \in C(D)$, as claimed. The image of $\rho_D$ is the kernel of $\gamma_D$, so the claim tells us that $\gamma_{D_{\mu}}P$ factors through $\gamma_D$. That is, there is a $\Lambda_{\mu}$-linear map $p:M_A(L) \to M_A(L-K_{\mu})$, with $\gamma_{D_{\mu}}P=p\gamma_D$. The image of $p$ includes $\gamma_{D_{\mu}}(a)$ for every $a \in A(D_{\mu})$; these elements generate $M_A(L-K_{\mu})$, so $p$ is surjective.

To complete the proof, we verify that $N=\ker p$. The inclusion $N \subseteq \ker p$ follows from two facts: $t_{\mu} \cdot x = x$ $\forall x \in M_A(L-K_{\mu})$, and $P(a)=0$ $\forall a \in A(D)$ with $\kappa_D(a)=\mu$. 

For the inclusion $N \supseteq \ker p$, notice that the equality $\gamma_{D_{\mu}}P=p\gamma_D$ implies
\[
\ker p = \gamma_D(\ker (\gamma_{D_{\mu}}P))=\gamma_D(P^{-1}(\ker \gamma_{D_{\mu}}))=\gamma_D(P^{-1}(\rho_{D_{\mu}}(\Lambda_{\mu-1}^{C(D_{\mu})})).
\]
For each crossing $c \in C(D_{\mu})$, $P(\rho_D(c))=\rho_{D_{\mu}}(c)$. Therefore $P^{-1}(\rho_{D_{\mu}}(\Lambda_{\mu-1}^{C(D_{\mu})}))$ is the submodule of $\Lambda_{\mu}^{A(D)}$ generated by $(\ker P) \cup \{\rho_D(c) \mid c \in C(D_{\mu}) \}$, and $\ker p$ is the image of this submodule under $\gamma_D$. Of course $\gamma_D(\rho_D(c))=0$ $\forall c \in C(D_{\mu})$, so it follows that $\ker p = \gamma_D(\ker P)$. 

We now claim that $\ker P$ is contained in the submodule of $\Lambda_{\mu}^{A(D)}$ generated by the set
\[
\widehat T=(t_{\mu}-1)\Lambda_{\mu}^{A(D)} \cup \{a \mid a \in A(D)\text{ and } \kappa_D(a)=\mu \}.
\]
If $x \in \ker P$ then as $\Lambda_{\mu}^{A(D)}$ is a free $\Lambda_{\mu}$-module, there is a unique function $f:A(D) \to \Lambda_{\mu}$ such that 
\[
x = \sum_{a \in A(D)} f(a)a \quad \text{and hence} \quad P(x)=\sum_{a \in A(D)} \pi f(a)P(a)=\sum_{a \in A(D_{\mu})} \pi f(a)a.
\]
As $\Lambda_{\mu-1}^{A(D_{\mu})}$ is a free $\Lambda_{\mu-1}$-module, $P(x)=0$ only if $\pi f(a)=0$ for each individual $a \in A(D_{\mu})$. This requires that for every $a \in A(D_{\mu})$,  $f(a)$ is an element of the ideal of $\Lambda_{\mu}$ generated by $t_{\mu}-1$. It follows that $x$ is an element of the submodule generated by $\widehat T$, as claimed.

As $\ker p = \gamma_D(\ker P)$, the claim implies that $\ker p \subseteq N$. The opposite inclusion was already verified, so $\ker p = N$. \end{proof}

Now, recall from \cite{mvaq1} that the quandle operations of $Q_A(L)$ are given by 
\[
x \triangleright y = (\phi_L(y)+1)x-\phi_L(x)y \quad \text{    and    } \quad x \triangleright^{-1} y=(\phi_L(y)+1)^{-1} \cdot (x+\phi_L(x)y).
\]
These operations are not restricted to $Q_A(L)$; $\triangleright$ is defined on all of $M_A(L)$, and $\triangleright^{-1}$ is defined whenever $\phi_L(y)+1$ is a unit of $\Lambda_{\mu}$. $Q_A(L)$ is the smallest subset of $M_A(L)$ that contains $\gamma_D(A(D))$ and is closed under the operations $\triangleright, \triangleright^{-1}$. An \emph{orbit} in $Q_A(L)$ is a minimal nonempty subset $X \subseteq Q_A(L)$ such that $x \triangleright y$, $x \triangleright^{-1}y \in X$ $\forall x \in X$ $\forall y \in Q_A(L)$. $Q_A(L)$ has $\mu$ orbits, one for each component of $L$. The $K_i$ orbit contains $\gamma_D(a)$ for every $a \in A(D)$ with $\kappa_D(a)=i$; we denote this orbit $Q_A(L)_i$. As 
\begin{align*}
&\quad \phi_L(x \triangleright y) = (\phi_L(y)+1)\phi_L(x)-\phi_L(x)\phi_L(y)=\phi_L(x) 
\\
& \text{and} \quad \phi_L(x \triangleright^{-1} y) = (\phi_L(y)+1)^{-1} \cdot (\phi_L(x)+\phi_L(x)\phi_L(y))
\\
&\qquad \qquad =(\phi_L(y)+1)^{-1} \cdot \phi_L(x) \cdot (1+\phi_L(y))=\phi_L(x) \text{,}
\end{align*}
the value of $\phi_L$ is constant on each orbit $Q_A(L)_i$; the constant value is $t_i-1$.

We say that the \emph{length} of an element $x \in Q_A(L)$ is $1$ more than the smallest number of applications of $\triangleright$ and $\triangleright^{-1}$ needed to obtain $x$ from elements of $\gamma_D(A(D))$. In particular, $x$ is of length $1$ if and only if  $x \in \gamma_D(A(D))$.
\begin{lemma}
\label{qgen}
Let $N$ be the $\Lambda_{\mu}$-submodule of $M_A(L)$ mentioned in Proposition \ref{induct}. Then $N$ is generated by the set $\widetilde T=(t_{\mu}-1)M_A(L) \cup Q_A(L)_{\mu}.$
\end{lemma}
\begin{proof}
As $Q_A(L)_{\mu}$ contains $\gamma_D(a)$ for every arc $a \in A(D)$ with $\kappa_D(a)=\mu$, $\widetilde T$ contains the set $T$ of Proposition \ref{induct}. 

If $x \in N$ is an element of $Q_A(L)_{\mu}$ and $y$ is any element of $Q_A(L)$, then as $(t_{\mu}-1)M_A(L) \subseteq N$,
\[
x \triangleright y = (\phi_L(y)+1)x-\phi_L(x)y = (\phi_L(y)+1)x-(t_{\mu}-1)y \in N \quad \text{  and}
\]
\[
x \triangleright^{-1} y=(\phi_L(y)+1)^{-1} \cdot (x+\phi_L(x)y)=(\phi_L(y)+1)^{-1} \cdot (x+(t_{\mu}-1)y) \in N.
\]
It follows that if $\ell \geq 1$ and $N$ contains all the elements of $Q_A(L)_{\mu}$ of length $\leq \ell$, then $N$ also contains all the elements of $Q_A(L)_{\mu}$ of length $ \ell +1$. As $N$ contains $\gamma_D(a)$ for every $a \in A(D)$ with $\kappa_D(a)=\mu$ -- that is, $N$ contains all the elements of $Q_A(L)_{\mu}$ of length $1$ -- it follows by induction that $N$ contains all the elements of $Q_A(L)_{\mu}$. Therefore $T \subseteq \widetilde T \subseteq N$.
\end{proof}

We deduce the following.
\begin{proposition}
\label{torres}
Let $L=K_1 \cup \dots \cup K_{\mu}$ and $L'=K'_1 \cup \dots \cup K'_{\mu'}$ be links, and let $f:Q_A(L) \to Q_A(L')$ be a quandle isomorphism. Proposition \ref{backref} tells us that $\mu=\mu'$, and we can re-index the components of $L$ and $L'$ so that $f$ extends to a Crowell equivalence $g:M_A(L) \to M_A(L')$. If $N \subseteq M_A(L)$ and $N' \subseteq M_A(L')$ are the submodules mentioned in Proposition \ref{induct} and Lemma \ref{qgen}, then $g(N)=N'$.
\end{proposition}
\begin{proof}
After $L$ and $L'$ are re-indexed compatibly with $f$, $f$ will map $Q_A(L)_i$ onto $Q_A(L')_i$, for each $i \in \{1, \dots, \mu \}$. Hence $g(Q_A(L)_{\mu})=Q_A(L')_{\mu}$. As $g$ is an isomorphism of $\Lambda_{\mu}$-modules, it must be that $g((t_{\mu}-1)M_A(L))=(t_{\mu}-1)M_A(L')$. It follows that if $\widetilde T \subseteq M_A(L)$ and $\widetilde T' \subseteq M_A(L')$ are the subsets discussed in Lemma \ref{qgen}, then $g(\widetilde T)= \widetilde T'$.
\end{proof}

\begin{lemma}
\label{phinduct} 
Let $p:M_A(L) \to M_A(L-K_{\mu})$ be the epimorphism of Proposition \ref{induct}. Then $\pi \phi_L = \phi_{(L-K_{\mu})}p:M_A(L) \to I_{\mu-1}$, $p(Q_A(L)_{\mu})= \{ 0 \}$, and $p(Q_A(L)_i)=Q_A(L-K_{\mu})_i$ for $1 \leq i < \mu$.
\end{lemma}
\begin{proof}
If $a \in A(D_{\mu})$ then $p(\gamma_D(a))=\gamma_{D_{\mu}}(a)$, so
\[
\pi \phi_L(\gamma_D(a)) = t_{\kappa_D(a)}-1 =  \phi_{(L-K_{\mu})}(\gamma_{D_{\mu}}(a))=\phi_{(L-K_{\mu})}(p(\gamma_D(a))).
\]
On the other hand, if $a \in A(D)$ has $\kappa_D(a)=\mu$ then $\gamma_D(a) \in N = \ker p$, so
\[
\pi \phi_L(\gamma_D(a)) = \pi(t_{\mu}-1) = 0 =  \phi_{(L-K_{\mu})}(0)=\phi_{(L-K_{\mu})}(p(\gamma_D(a))).
\]

We see that $\pi \phi_L(\gamma_D(a))=\phi_{(L-K_{\mu})}(p(\gamma_D(a)))$ $\forall a \in A(D)$. As $M_A(L)$ is generated by $\gamma_D(A(D))$, it follows that $\pi \phi_L = \phi_{(L-K_{\mu})}p$.

Of course $p(Q_A(L)_{\mu})=\{0\}$ follows immediately from Lemma \ref{qgen}.

Now, notice that if $x,y \in M_A(L)$ then 
\begin{align*}
& p(x \triangleright y) = p((\phi_L(y)+1)x-\phi_L(x)y)= p((\phi_L(y)+1)x)-p(\phi_L(x)y)
\\
& \qquad = \pi (\phi_L(y)+1) p(x)-\pi (\phi_L(x))p(y)
\\
& \qquad =  (\phi_{(L-K_{\mu})}(p(y)) +1) p(x) -\phi_{(L-K_{\mu})}(p(x))p(y)=p(x) \triangleright p(y) \text{,} 
\end{align*}
and if $\phi_L(y)+1$ is a unit of $\Lambda_{\mu}$ then
\begin{align*}
& p(x \triangleright^{-1} y)=p((\phi_L(y)+1)^{-1} \cdot (x+\phi_L(x)y))
\\
& \qquad =\pi((\phi_L(y)+1)^{-1}) \cdot (p(x)+\pi(\phi_L(x))p(y))
\\
& \qquad =(\pi(\phi_L(y))+1)^{-1} \cdot (p(x)+\pi(\phi_L(x))p(y))
\\
& \qquad =(\phi_{(L-K_{\mu})}(p(y))+1)^{-1} \cdot (p(x)+\phi_{(L-K_{\mu})}(p(x))p(y)) = p(x) \triangleright^{-1} p(y).
\end{align*}
That is, $p$ is a homomorphism of the operations $\triangleright$ and $\triangleright^{-1}$.

For each $i \in \{1, \dots, \mu-1 \}$, $Q_A(L)_i$ is the smallest subset of $M_A(L)$ that contains $\{\gamma_D(a) \in A(D) \mid \kappa_D(a)=i \}$ and has $x \triangleright \gamma_D(a) , x \triangleright^{-1} \gamma_D(a) \in Q_A(L)_i$ $\forall x \in Q_A(L)_i$ $\forall a \in A(D)$. The orbit $Q_A(L-K_{\mu})_i$ is described in a similar way, using $D_{\mu}$ rather than $D$. As $p( \gamma_D(a))= \gamma_{D_{\mu}}(a)$ $\forall a \in A(D_{\mu})$ and $p$ is a homomorphism of $\triangleright$ and $\triangleright^{-1}$, it follows that $p(Q_A(L)_i)=Q_A(L-K_{\mu})_i$. \end{proof}

\begin{corollary}
\label{qtorres}
Suppose $f:Q_A(L) \to Q_A(L')$ is a quandle isomorphism, and $L$ and $L'$ are indexed compatibly with $f$. Then $Q_A(L-K_{j}) \cong Q_A(L'-K'_{j})$ for every $j \in \{1, \dots, \mu \}$.
\end{corollary}
\begin{proof} It suffices to verify the corollary when $j=\mu$. According to Propositions \ref{backref} and \ref{torres}, there is a Crowell equivalence $g:M_A(L) \to M_A(L')$, which extends $f$ and has $g(N)=N'$. Then $g$ induces an isomorphism $h:M_A(L)/N \to M_A(L')/N'$. The epimorphisms $p:M_A(L) \to M_A(L-K_{\mu})$ and $p':M_A(L') \to M_A(L'-K'_{\mu})$ of Proposition \ref{induct} have kernels $N$ and $N'$, so they induce isomorphisms $q:M_A(L)/N \to M_A(L-K_{\mu})$ and $q':M_A(L')/N' \to M_A(L'-K'_{\mu})$.

Then $q'hq^{-1}:M_A(L-K_{\mu}) \to M_A(L'-K'_{\mu})$ is an isomorphism of $\Lambda_{\mu}$-modules. We claim that $q'hq^{-1}$ restricts to a quandle isomorphism between $Q_A(L-K_{\mu})$ and $Q_A(L'-K'_{\mu})$. 

To verify the claim, note first that if $1 \leq i < \mu$ then $q'hq^{-1}(Q_A(L-K_{\mu})_i)=p'gp^{-1}(Q_A(L-K_{\mu})_i)$. By Proposition \ref{induct} and Lemma \ref{phinduct}, 
\[
p^{-1}(Q_A(L-K_{\mu})_i)=(\ker p) + Q_A(L)_i=N+ Q_A(L)_i.
\]
As $L$ and $L'$ have been indexed compatibly with $f$, and $g$ extends $f$, $g(Q_A(L)_i)=f(Q_A(L)_i)=Q_A(L')_i$. Proposition \ref{torres} tells us that $g(N)=N'$, so
\[
gp^{-1}(Q_A(L-K_{\mu})_i)=g(N+ Q_A(L)_i)=N'+Q_A(L')_i.
\]
Applying Proposition \ref{induct} and Lemma \ref{phinduct} to $p'$, we conclude that
\[
q'hq^{-1}(Q_A(L-K_{\mu})_i)=p'gp^{-1}(Q_A(L-K_{\mu})_i)
\]
\[
=p'(N'+Q_A(L')_i)=p'(Q_A(L')_i)=Q_A(L'-K'_{\mu})_i.
\]
This verifies part of the claim: $q'hq^{-1}$ restricts to a bijection between $Q_A(L-K_{\mu})$ and $Q_A(L'-K'_{\mu})$.

To complete the proof of the claim, recall that the proof of Lemma \ref{phinduct} includes the equalities $p(x \triangleright y)=p(x) \triangleright p(y)$ (valid for all $x,y \in M_A(L)$) and $p(x \triangleright^{-1} y)=p(x) \triangleright^{-1} p(y)$ (valid so long as $\phi_L(y)+1$ is a unit). Also $\phi_{L'}g=\phi_L$, because $g$ is a Crowell equivalence. Hence if $x,y \in M_A(L)$ then
\[
g(x \triangleright y)= g((\phi_L(y)+1)x-\phi_L(x)y)= (\phi_L(y)+1)g(x)-\phi_L(x)g(y)
\]
\[
= (\phi_{L'}(g(y))+1)g(x)-\phi_{L'}(g(x))g(y)=g(x) \triangleright g(y)
\]
and if $\phi_L(y)+1$ is a unit,
\[
g(x \triangleright^{-1} y)= g((\phi_L(y)+1)^{-1} \cdot (x+\phi_L(x)y))= (\phi_L(y)+1)^{-1} \cdot (g(x)+\phi_L(x)g(y))
\]
\[
= (\phi_{L'}(g(y))+1)^{-1} \cdot (g(x)+\phi_{L'}(g(x))g(y))=g(x) \triangleright^{-1} g(y).
\]
It follows that the restriction of $q'hq^{-1}$ to a map $Q_A(L-K_{\mu}) \to Q_A(L'-K'_{\mu})$ is a quandle homomorphism. It's bijective, so it's an isomorphism.
\end{proof}

Theorem \ref{main} follows from Corollary \ref{qtorres}, using induction.

\section{Distinguishing \texorpdfstring{$W$}{W} from \texorpdfstring{$7^2_8$}{7} using colorings}
\label{moreexamples}

The fact that $W$ and $L=7^2_8$ have non-isomorphic $Q_A$ quandles follows immediately from Proposition \ref{backref} and Theorem \ref{main}, as the components of the two links have non-isomorphic Alexander modules. For the purpose of illustration, though, we present in this section a direct proof of the fact that even though $M_A(W)$ and $M_A(L)$ are isomorphic, there is no isomorphism $f:M_A(W) \to M_A(L)$ with $f(Q_A(W))=Q_A(L)$. This proves that $Q_A(W) \not \cong Q_A(L)$ because according to Proposition \ref{backref}, if $Q_A(W)$ and $Q_A(L)$ were isomorphic, an isomorphism between them would extend to an isomorphism between the Alexander modules. 

The idea of the direct proof is to detect the difference between the trivial components of $W$ and the trefoil component of $L$ using multivariate Alexander colorings, i.e., $\Lambda_2$-module homomorphisms with $M_A(W)$ and $M_A(L)$ as domains \cite{Tcol}. Recall that a Fox coloring of a knot is, in essence, a homomorphism from the knot's Alexander module to a $\Lambda_{1}$-module $M$ such that $(t_1+1) M = 0$. An unknot has determinant 1 and a trefoil has determinant 3, so if $3M=0 \neq M$ the trefoil will have some nonconstant Fox colorings in $M$, but the unknot will have none. (The connection between Fox colorings and the determinant is well known; see \cite{N} for instance.) Now, suppose $M$ is a nontrivial $\Lambda_2$-module with $(t_1+1) M = 3M = (t_2-1)M = 0$. Then according to Lemma \ref{qgen}, a $\Lambda_2$-linear map $M_A(W) \to M$ whose kernel contains $Q_A(W)_2$ will induce a homomorphism $M_A(K_1) \to M$, where $K_1$ is the first component of $W$; the same reasoning applies to $L=7^2_8$, of course. This reasoning predicts that $L$ will have multivariate Alexander colorings in $M$ that are 0 on the $K_2$ orbit of $Q_A(L)$, and nonconstant on the $K_1$ orbit; but $W$ will have no such coloring.

We proceed to verify these predictions in detail. Let $GF(3)$ be the field with three elements, let $\chi:\Lambda_2 \to GF(3)$ be the homomorphism of rings with unity given by $\chi(t_1) = -1$ and $\chi(t_2)=1$, and let $GF(3)_{\chi}$ be the $\Lambda_2$-module obtained from $GF(3)$ using $\chi$. That is, $\lambda \cdot x = \chi(\lambda) \cdot x$ $\forall \lambda \in \Lambda_2$ $\forall x \in GF(3)$.

\begin{proposition}
\label{wprop}
Let $g:M_A(W) \to GF(3)_{\chi}$ be a  $\Lambda_2$-linear map that is constant on the orbit $Q_A(W)_2$. Then $g$ is also constant on the orbit $Q_A(W)_1$.
\end{proposition}
\begin{proof}
Let $D$ be the diagram of $W$ pictured in Fig.\ \ref{wfig}. For convenience we write $g(\gamma_D(a_i))=g_i$ for $1 \leq i \leq 5$. As $g$ is $\Lambda_2$-linear, every crossing $c$ of $D$ has $g(\gamma_D(\rho_D(c)))=0$. Considering the crossings of $D$ in order, we conclude that the following elements of $GF(3)$ are all $0$.
\[
-g_3-g_1,-g_1+g_2-g_4,-g_3-g_1,-g_3+g_5-g_4, \text{ and }   -g_4-g_5-g_2
\]
As $g$ is constant on $Q_A(W)_2$, $g_1=g_3$. Then $-g_3-g_1=0$ implies $0=-2g_1$. As $-2=1$ in $GF(3)$, it follows that the constant value of $g$ on $Q_A(W)_2$ is $0$. Therefore the second and fourth elements displayed above are $g_2-g_4$ and $g_5-g_4$; they equal $0$, so $g_2=g_4=g_5$. That is, $g$ is constant on the length $1$ elements of $Q_A(W)_1$.

The argument proceeds using induction on length. Suppose $\ell \geq 1$ and $g$ is constant on the elements of $Q_A(W)_1$ of length $\leq \ell$. Suppose $x$ and $y$ are elements of $Q_A(W)_1$, with $g(x)=g(y)=g_2$. Then 
\[
g(x \triangleright y)=g((\phi_W(y)+1)x-\phi_W(x)y)
=g(t_1x-(t_1-1)y)
\]
\[
=t_1g(x)-(t_1-1)g(y)=\chi(t_1-(t_1-1)) \cdot g_2 = 1 \cdot g_2 = g_2 \text{  and}
\]
\[
g(x \triangleright^{-1} y)=(\phi_W(y)+1)^{-1} \cdot (g(x)+\phi_W(x)g(y))
\]
\[
=(\phi_W(y)+1)^{-1} (1+\phi_W(x)) \cdot g_2 = 1 \cdot g_2 = g_2.
\]

Also, if $x$ is an element of $Q_A(W)_1$ with $g(x)=g_2$ and $y$ is an element of $Q_A(W)_2$ then $y \in \ker g$, so
\[
g(x \triangleright y)=g((\phi_W(y)+1)x-\phi_W(x)y)
=g(t_2x-(t_1-1)y)
\]
\[
=t_2g(x)-(t_1-1)g(y)=\chi(t_2) \cdot g(x)-\chi(t_1-1) \cdot 0=1 \cdot g_2 = g_2 \text{  and}
\]
\[
g(x \triangleright^{-1} y)=(\phi_W(y)+1)^{-1} \cdot (g(x)+\phi_W(x)g(y))
\]
\[
=t_2^{-1} \cdot (g_2 + (t_1-1) \cdot 0) = \chi(t_2^{-1}) \cdot g_2=1 \cdot g_2 = g_2.
\]
\end{proof}

For the sake of argument, suppose there is a $\Lambda_2$-module isomorphism $f:M_A(W) \to M_A(L)$ such that $f(Q_A(W))=Q_A(L)$. According to Proposition \ref{backref}, the components of $W$ may be re-indexed compatibly with $f$. As the components of $W$ are interchanged by a symmetry of the link, we may presume that $f$ is a Crowell equivalence with the component indices indicated in Figs.\ \ref{wfig} and \ref{otherfig}. Then $\phi_Lf=\phi_W$, so $f$ maps $Q_A(W)_1$ to $Q_A(L)_1$, and $f$ maps $Q_A(W)_2$ to $Q_A(L)_2$. It follows that Proposition \ref{wprop} applies to $M_A(L)$ and $Q_A(L)$. 

In reality, though, Proposition \ref{wprop} does not apply to $M_A(L)$ and $Q_A(L)$. Refer to the diagram $E$ pictured in Fig.\ \ref{otherfig}. The values $g(\gamma_E(a_1))=g(\gamma_E(a_3))=g(\gamma_E(a_6))=g(\gamma_E(a_7))=0$, $g(\gamma_E(a_2))=g(\gamma_E(a_4))=1$ and $g(\gamma_E(a_5))=-1$ satisfy all the crossing relations from $E$, so they define a $\Lambda_2$-linear map $g:M_A(L) \to GF(3)_{\chi}$. As $g(\gamma_E(a_2)) \neq g(\gamma_E(a_4))$, $g$ is not constant on $Q_A(L)_1$. However, an inductive argument much like the proof of Proposition \ref{wprop} can be used to verify that $\ker g$ contains $Q_A(L)_2$. 

There are other ways to use these ideas to show that $Q_A(W) \not \cong Q_A(L)$. For instance, the reader will have little trouble showing that Proposition \ref{wprop} holds for $L$ if the component indices in $L$ are interchanged. As Proposition \ref{wprop} does not hold for $L$ with its original component indices, it follows that there is no automorphism of $Q_A(L)$ that interchanges $Q_A(L)_1$ and $Q_A(L)_2$. The quandle $Q_A(W)$, instead, certainly has an automorphism that interchanges the two orbits, because the link $W$ has a symmetry that interchanges the two components. Thus $Q_A(L)$ and $Q_A(W)$ are distinguished by the actions of their automorphism groups.


\begin{thebibliography}{99}

\bibitem{A} J. W. Alexander, Topological invariants of knots and links, \emph{Trans. Amer. Math. Soc.} \textbf{30} (1928) 275-306. 

\bibitem{C1} R. H. Crowell, Corresponding link and module sequences, \emph{Nagoya Math. J.} \textbf{19} (1961) 27-40.

\bibitem{C3} R. H. Crowell, The derived module of a homomorphism, \emph{Adv. in Math.} \textbf{6} (1971) 210-238.

\bibitem{CS} R. H. Crowell and D. Strauss, On the elementary ideals of link modules, \emph{Trans. Amer. Math. Soc.} \textbf{142} (1969) 93-109. 

\bibitem {F} R. H. Fox, A quick trip through knot theory, in \emph{Topology of 3-Manifolds and Related Topics} (Proc. The Univ. of Georgia Institute, 1961) (Prentice-Hall, Englewood Cliffs, N.J., 1962), pp. 120-167.

\bibitem {H} J. A. Hillman, \emph{Algebraic Invariants of Links}, 2nd edn. Series on Knots and Everything, Vol. 52 (World Scientific, Singapore, 2012).

\bibitem {J} D. Joyce, A classifying invariant of knots, the knot quandle, \emph{J. Pure Appl. Algebra} \textbf{23} (1982) 37-65.

\bibitem {L}W. B. R. Lickorish, \emph{An Introduction to Knot Theory}, Graduate Texts
in Mathematics, Vol. 175 (Springer, New York, 1997).

\bibitem {M} S. V. Matveev, Distributive groupoids in knot theory, \emph{Mat. Sb. (N.S.)} \textbf{119} (1982) 78-88.

\bibitem {N} O. Nanyes, Link colorability, covering spaces and isotopy, \emph{J. Knot Theory Ramifications} \textbf{6} (1997) 833-849.

\bibitem{Ro} D. Rolfsen, \emph{Knots and Links} (Publish or Perish, Berkeley, 1976).

\bibitem{Sa} N. Sato, Alexander modules of sublinks and an invariant of classical link concordance, \emph{Illinois J. Math.} \textbf{25} (1981) 508-519.

\bibitem{To} G. Torres, On the Alexander polynomial, \emph{Ann. of Math. (2)} \textbf{57} (1953) 57-89.

\bibitem{Tr}L. Traldi, A generalization of Torres' second relation, \emph{Trans. Amer. Math. Soc.} \textbf{269} (1982) 593-610. 

\bibitem{Tcol} L. Traldi, Multivariate Alexander colorings, \emph{J. Knot Theory Ramifications} \textbf{27} (2018) 1850076.

\bibitem{mvaq1} L. Traldi, Multivariate Alexander quandles, I. The module sequence of a link, \emph{J. Knot Theory Ramifications}, to appear.

\bibitem{mvaq2} L. Traldi, Multivariate Alexander quandles, II. The involutory medial quandle of a link, arxiv:1902.10603.

\end{thebibliography}
\end{document}